\documentclass[11pt]{amsart}
\usepackage{amsmath,amsthm,amsfonts,amssymb,amscd,flafter,pinlabel,float,comment}
\usepackage[mathscr]{eucal}
\usepackage{graphics}

\usepackage[usenames,dvipsnames]{color}
\usepackage{xcolor}
\usepackage{subfigure}
\usepackage{graphicx}

\usepackage{pgf}
\usepackage{tikz}
\usetikzlibrary{arrows,automata}
\usepackage[latin1]{inputenc}

\usepackage{tikz-cd}
\usetikzlibrary{intersections,shapes}
\usetikzlibrary{decorations.markings,knots}
\usetikzlibrary{decorations.pathreplacing, decorations.pathreplacing} 
\usetikzlibrary{arrows.meta,backgrounds,positioning,fit,petri, matrix,calc}

\newtheorem{theorem}{Theorem}[section]
\newtheorem{corollary}[theorem]{Corollary}

\newtheorem {example}[theorem]{Example}

\theoremstyle{definition}
\newtheorem{definition}{Definition}[section]

\theoremstyle{definition}
\newtheorem{remark}[theorem]{Remark}

\title[$n$-Fold Cyclic Branched Covers and Overtwisted Contact Structures]{$n$-Fold Cyclic Branched Covers and Overtwisted Contact Structures}

\author[J. Ceniceros]{Jose Ceniceros}
\address{Hamilton College, Clinton, NY, USA}
\email{jcenicer@hamilton.edu}

\begin{document}

\maketitle

\begin{abstract}
This article presents an alternate way to prove a result originally proven by Harvey, Kawamuro, and Plamenvskaya in \cite{HaKaPl}. We accomplish this by explicitly constructing an overtwisted disk in the $n$-fold cyclic branched cover of $S^3$ with the standard contact structure branched along a carefully selected transverse knot. Furthermore, by utilizing this construction of the overtwisted disk, we can see that the overtwisted disk is contained in the complement of the branch locus.
\end{abstract}

\section{Introduction}\label{intro}

This article provides an alternate proof of a result of Harvey, Kawamuro, and Plamenvskaya in \cite{HaKaPl}. The result classifies the contact structures obtained from $n$-fold cyclic branched covers of $\mathbb{R}^3$ branched along specific transverse knots. Contact structures on a 3-manifold are classified as tight or overtwisted. Classifying contact structures on a 3-manifold has become a central question in contact topology. In \cite{El1}, Eliashberg showed that overtwisted contact structures on a compact 3-manifold are classified by their homotopy type. In addition to classifying contact structures on a 3-manifold, there has also been interest in classifying specific submanifolds of a contact manifold. The submanifolds are known as Legendrian and transverse knots. The study of these objects has become interesting in their own right, and they have gained the attention of low-dimensional topologists \cite{Ar}. Eliashberg and Frasher proved that the Legendrian knot is completely classified by the Thurston-Bennequin invariant and rotation number \cite{ElFr}. Etnyre and Honda proved that the torus and figure-eight knots are all completely classified by the Thurston-Bennequin invariant and rotation number \cite{EtHo}. 

A natural construction in topology is the branched cover of a 3-manifold. The interest in the relationship between 3-manifolds and branched coverings was sparked by the following result by Alexander \cite{Al}. 
\begin{theorem}
Every closed orientable 3-manifold is a branched covering space of $S^3$ with branch locus a link in $S^3$. 
\end{theorem}
Alexander's result illuminated the profound relationship between links, branched covers, and 3-manifolds. Over time, his findings have been refined, with Hilden and Montesinos contributing refinements that state that every closed orientable 3-manifold is a 3-fold branched covering branched along a knot \cite{Hi,Mo}. Furthermore, there have also been restrictions to branched covers looking at a fixed branch locus and still obtaining all closed-oriented 3-manifolds. A link $K$ is called \textit{universal} if every 3-manifolds is obtained as a branched cover branching along K. In \cite{Th}, Thurston proved the existence of a universal link. Since then, it has been shown that the figure-eight knot, Borromean rings, and Whitehead link and $9_{46}$ are also universal \cite{HiLoMo}.  

The study of 3-manifolds has greatly benefited from the study of branched covers.  Gonzalo in \cite{Go} carefully extended the definition of branched covers to the setting of contact structures. Furthermore, in \cite{HaKaPl}, Harvey, Kawamuro, and Plamenevskaya studied contact manifolds that arise from $n$-fold cyclic branched covers branched along transverse knots in $(S^3, \xi_{std})$. The authors of that paper were interested in the following question: Suppose that transverse knots $K_1$, $K_2$ are smoothly isotopic, and $sl(K_1) = sl(K_2)$. If  $p \geq 2$ is fixed, will the $n$-fold cyclic covers branched over $K_1$ and $K_2$ be contactomorphic? Although the authors were not able to prove that all of the contact manifolds arising from branched covers are contactomorphic, they were able to prove for several transverse knots that the branched covers are contactomorphic.  More interestingly, they were able to show that tight contact structures give rise to overtwisted contact structures when considering the $n$-cyclic branched cover branched along carefully chosen transverse knots. Specifically, they showed that $n$-fold cyclic branched cover is overtwisted if the branch locus $L$ is obtained as a transverse stabilization of another transverse link. We reprove this result using a construction of an overtwisted disk introduced by Gompf in \cite{Gompf1}. This construction, used in Proposition 5.1 of \cite{Gompf1}, provides a method to distinguish fillable contact structures. Using this alternate construction, we show that the overtwisted disk is completely contained in the complement of the branch locus. 

This article is organized as follows. Section~\ref{contact} reviews the necessary definitions and results regarding contact manifolds. In Section~\ref{LegTrans}, we provide basic definitions and results of Legendrian and transverse knots. In Section~\ref{Classical}, we define the classical invariants for Legendrian knots and transverse knots. In the same section, we include a well-known formula for computing the Thurston-Bennequin invariant directly from the front projection of a Legendrian knot. In Section~\ref{Braids}, we present the connection between transverse knots and braids. In Section~\ref{Branched}, we recall the definition of $n$-fold branched covers of a contact manifold branched along a transverse knot. Lastly, in Section~\ref{Proof}, we reprove the result of Harvey, Kawamuro, and Plamenvskaya regarding overtwisted contact structures resulting from the $n$-fold branched cover construction.  

\section{Contact Structures}\label{contact}
This section will introduce the basic definition and results regarding contact structures on a $ 3$ manifold. This section aims to give an overview of contact geometry and only cover results relevant to the result in this paper. For a complete description of the theory, the reader is referred to \cite{Et2, Ge}. We begin with the following definition.

\begin{definition}
An oriented 2-plane field $\xi$ on a 3-manifold $Y$ is called a contact structure if there exists a 1-form $\alpha \wedge d \alpha \neq 0$. The pair $(Y, \xi)$ is called a \emph{contact manifold}. 
\end{definition}

The condition $\alpha \wedge d\alpha \neq 0$ is known as a non-integrability condition. This condition ensures that there is no embedded surface in $Y$, which is tangent to $\xi$ on any open neighborhood. The following are three well-known examples of contact structures defined on $\mathbb{R}^3$.

\begin{example}\label{stdcon}
Consider the 3-manifold $Y=\mathbb{R}^3$ with standard Cartesian coordinates $(x,y,z)$ and the 1-form
\[ \alpha = dz -y dx \]
with a simple computation we can confirm that $\alpha \wedge d \alpha \neq 0$. Thus, $\alpha$ is a contact form and $\xi_{std} = \mathrm{ker}(\alpha) = \mathrm{ker}(dz -ydx)$ is a contact structure on $\mathbb{R}^3$. 	
\end{example}

\begin{remark}
Example \ref{stdcon} is commonly referred to as the standard contact structure on $\mathbb{R}^3$. At any point on the $xz$-plane, the plane field is horizontal. Furthermore, moving along a ray perpendicular to the $xz$-plane, the plane field will always be tangent to this ray and rotate by $\pi/2$ in a left-handed manner along the ray.
\end{remark}

\begin{example}
\label{rotcon}
Consider the 3-manifold $Y = \mathbb{R}^3$ with cylindrical coordinates $(r, \theta, z)$ and the 1-form 
\[ \alpha = dz + r^2d\theta. \]
Again, with a simple computation, we can check that $\alpha \wedge d \alpha \neq 0$. So $\xi_{sym} =\mathrm{ker}(\alpha) = \mathrm{ker}(dz + r^2d\theta)$ is a contact structure on $M$. 
\end{example}

\begin{remark}
In this contact manifold, the contact planes twist clockwise as you move along any ray perpendicular to the $z$-axis. At the $z$-axis, the contact planes are horizontal. 
\end{remark}

\begin{example}
\label{overcon}
Consider the 3-manifold $Y = \mathbb{R}^3$ with cylindrical coordinates $(r, \theta, z)$ and the 1-form 
\[ \alpha = \cos(r)dz + r \sin(r)d\theta. \]
Again, with a simple computation, we can check that $\alpha \wedge d \alpha \neq 0$. So $\xi_{OT} =\mathrm{ker}(\alpha) = \mathrm{ker}(\cos(r)dz + r \sin(r)d\theta)$ is a contact structure on $M$. 
\end{example}

\begin{remark}
In this case, we see that $\xi_{OT}$ is horizontal along the $z$-axis, and as you move out on any ray perpendicular to the $z$-axis, the planes will twist in a clockwise manner, and the planes will make infinitely many full twists as $r$ goes towards infinity. 
\end{remark}

Contact structures fall into two disjoint classes: tight and overtwisted. A contact manifold $(Y, \xi)$ is \textit{overtwisted} if there exists an embedded disk $D \in Y$ such that $T_pD = \xi_p$ for every $p \in \partial D$. This means that the contact planes are tangent to the disk's boundary $D$. On the other hand, the contact manifold is said to be \textit{tight} if it is not overtwisted. It is only natural to classify contact structures. We have the following notion of contact manifolds being equivalent. 

\begin{definition}
Two contact structures $\xi_0$ and $\xi_1$ on a manifold $Y$ are called contactomorphic if there is a diffeomorphism $f: Y \rightarrow Y$ such that $f$ send $\xi_0$ to $\xi_1$: 
\[ f_*(\xi_0) = \xi_1. \]
\end{definition}

The contact structures in Example \ref{stdcon} and Example \ref{rotcon} are contactomorphic. The contact structure defined in Example 2.4 is not contactomorphic to those defined in Example 2.2 and  Example 2.4. An explicit contactomorphism between Example \ref{stdcon} and Example \ref{rotcon} can be found in \cite{Ge}.

\section{Legendrian and Transverse Knots}\label{LegTrans}

This section will introduce a few basic definitions and results for Legendrian and transverse knots. This section is not meant to be a complete survey on the subject; for a detailed description of knot theory supported in a contact 3-manifold, see \cite{Et1,Ge}. We will also assume that the reader has some knowledge of topological knots. For a detailed treatment of this topic, see \cite{Ro}. 

There are two distinct types of knots that adhere to the geometry imposed by contact structures: Legendrian and transverse knots.

\begin{definition}
A Legendrian knot $L$ in a contact 3-manifold $(Y, \xi)$ is an 
embedded $S^1$ that is always tangent to $\xi$:
\[ T_x L \in \xi_x, \quad x \in L.\]
\end{definition}

Two Legendrian knots in $(\mathbb{R}^3, \xi_{std})$ are \textit{Legendrian isotopic} if they can be deformed into each other through a continuous family of Legendrian knots. In this article, we focus on the \textit{front projection} of a Legendrian knot, denoted by $\Pi(L)$, which is its projection onto $\mathbb{R}^2$ given by $(x, y, z) \mapsto (x, z)$.

For a Legendrian knot $L$, the front projection $\Pi(L)$ captures essential features of $L$ useful in studying Legendrian knots. This projection must avoid vertical tangencies, have cusps as the only non-smooth points, and ensure that at each crossing, the slope of the overcrossing is less than that of the undercrossing.

By adhering to these conditions, the front projection provides a valuable tool for studying Legendrian knots in $\mathbb{R}^3$.

An interesting note about Legendrian knots is that you have different Legendrian knot representatives of a topological knot type. The operation that produces different Legendrian knots of the same topological knot type is called \textit{stabilization}. A stabilization of a Legendrian knot $L$ in a front projection of $L$ can be obtained by removing a strand and replacing it with a zig-zag, see Figure~\ref{LegStab}. We denote a positive stabilization by $S_{+}$ and a negative stabilization by $S_{-}$. 

\begin{figure}
    \centering 
    \begin{tikzpicture}[
node distance = 3mm,
image/.style = {scale=0.3,},
arr/.style = {-Triangle, semithick}
]
\node[image] (A) {\includegraphics[scale=.2]{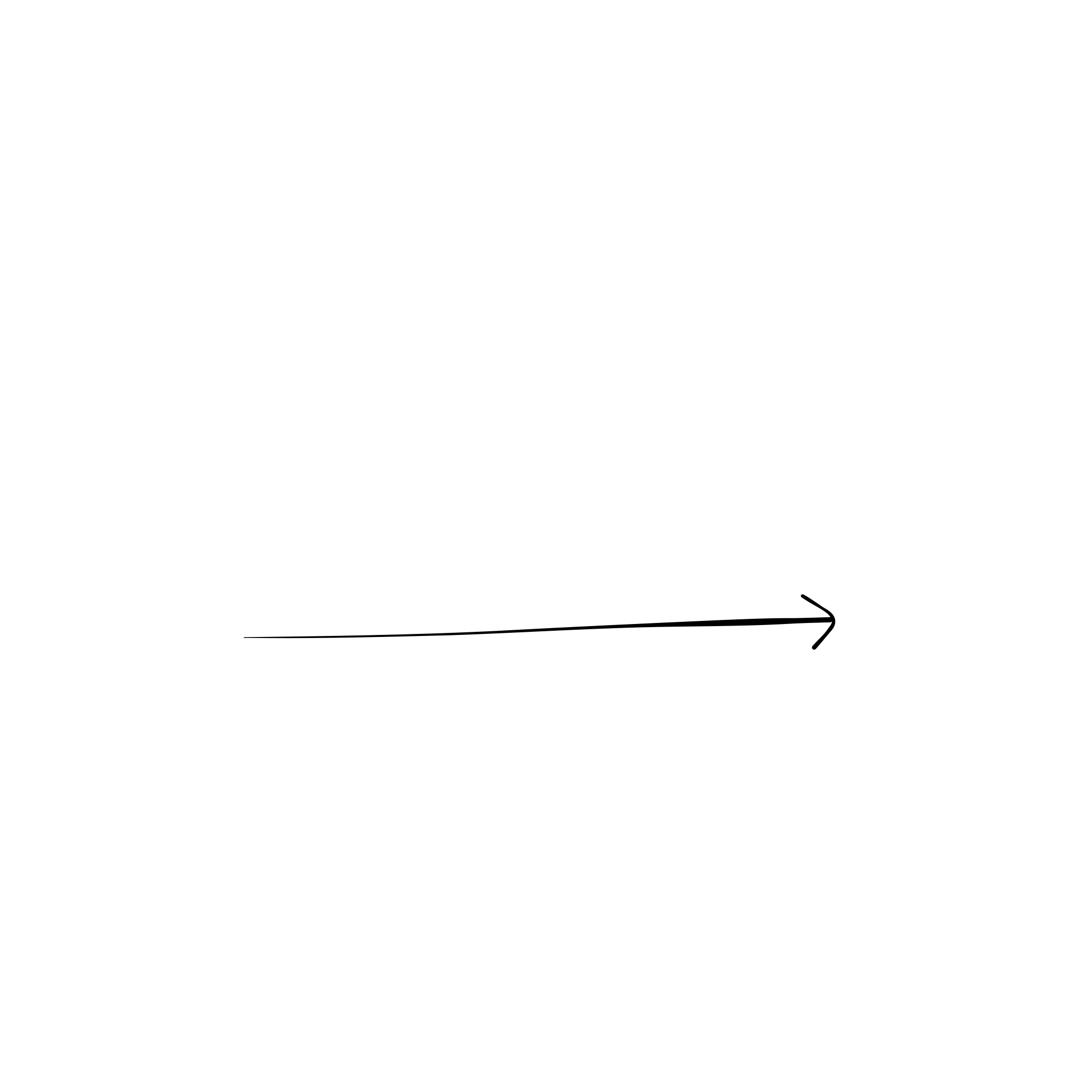}};
\node[image,right= 4cm of A] (D){};
\node[image, below= 0mm of D] (C) {\includegraphics[scale=.2]{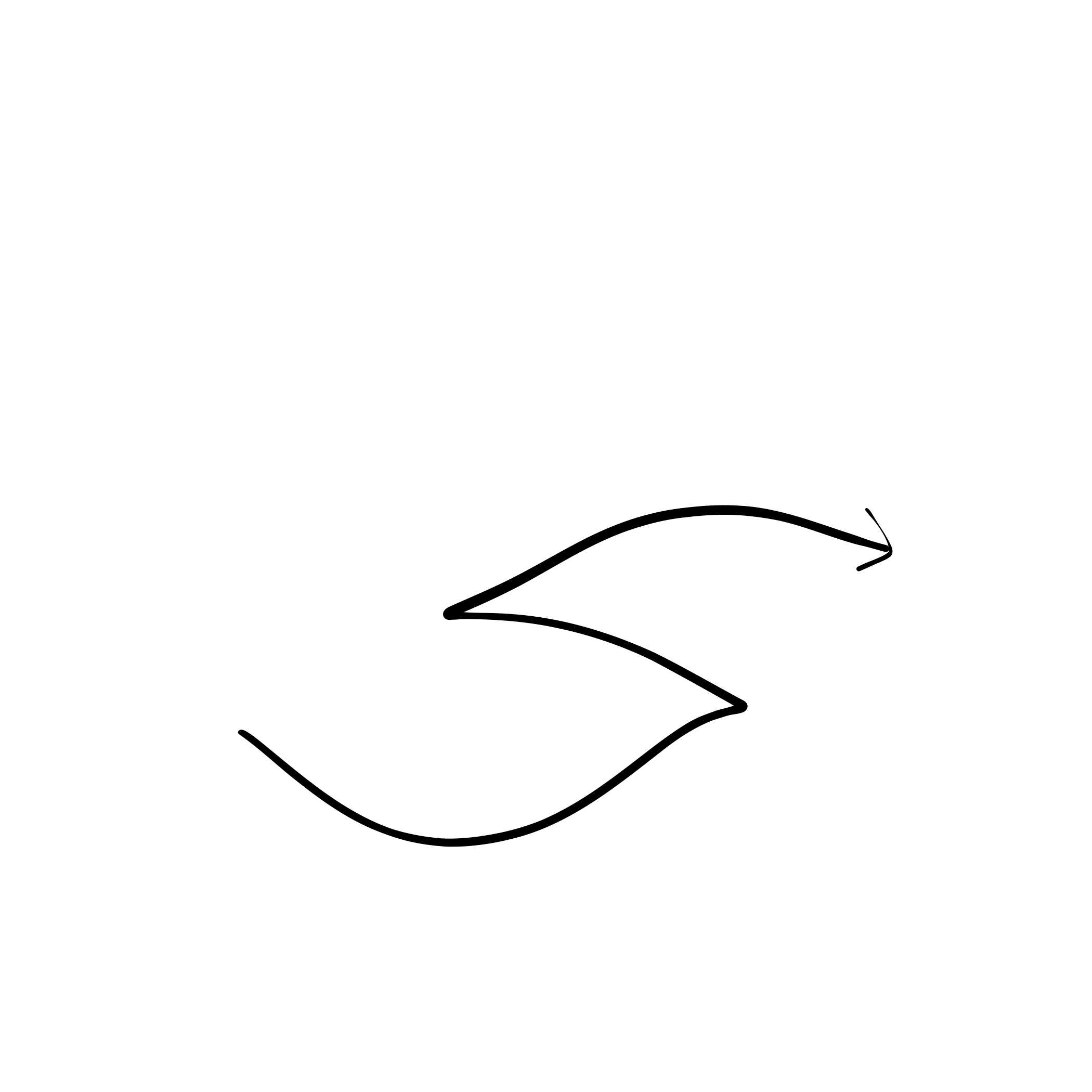}};
\node[image,above= 0mm of D] (B) {\includegraphics[scale=.2]{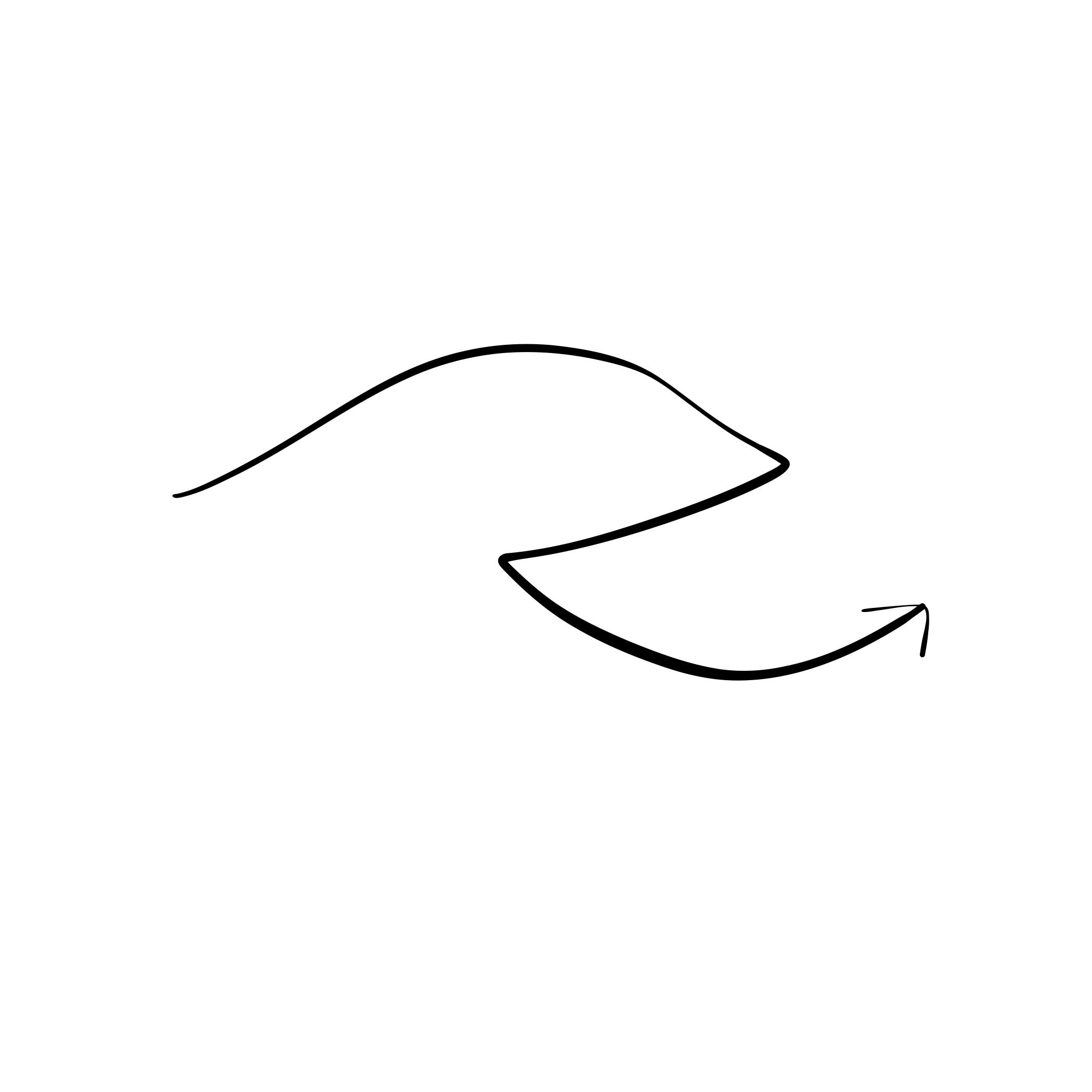}};
\draw[line width=0.30mm, ->](A) --(B);
\draw[line width=0.30mm, ->](A) --(C);
\node [draw=none, inner sep = 0] at (2,1) {\large{$S_{+}$}};
\node [draw=none, inner sep = 0] at (2,-1) {\large{$S_{-}$}};
\end{tikzpicture}
    \caption{A stabilization of a Legendrian knot $L$.}
    \label{LegStab}
\end{figure}
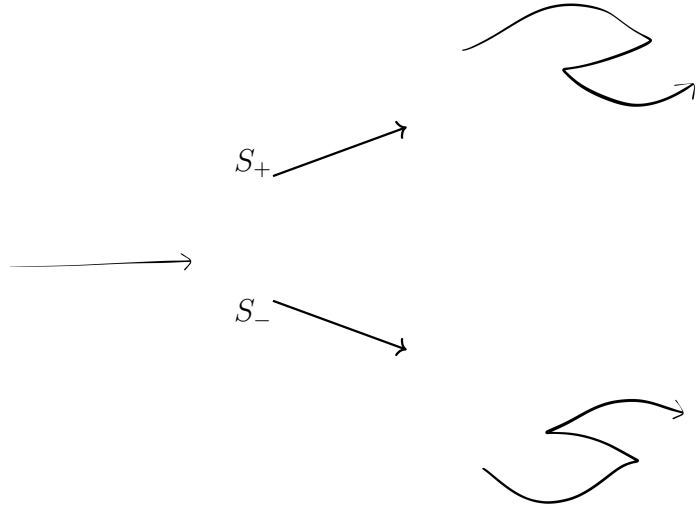

\begin{definition}
A transverse knot $K$ in a contact manifold $(Y, \xi)$ is an 
embedded $S^1$ that is always transverse to $\xi$:
\[ T_x K \oplus \xi_x = T_x Y, \quad x \in K. \]
\end{definition}

We classify transverse knots up to transverse isotopy. We say that two transverse knots are \textit{transversely isotopic} if an isotopy takes one knot to the other while remaining transverse to the contact planes. 

Similar to the case of Legendrian knots, transverse knots can be studied using front projects. The definition of the front projection of transverse knots is similar in spirit to the front projection of Legendring knots. For a description of transverse knots, see \cite{Et}. There is also the notation of stabilization for transverse knots. In the case of transverse knots, the stabilization is formed by taking an arc in the front projection, see the left figure in Figure~\ref{TransStab}, and replacing it with an arc illustrated in the right figure of Figure~\ref{TransStab}. We will denote the stabilization of the transverse knot $K$ by $K_{stab}$. 

\begin{figure}
    \centering 
    \begin{tikzpicture}[
node distance = 2mm,
image/.style = {scale=0.4,},
arr/.style = {-Triangle, semithick}
]
\node[image] (A) {\includegraphics[scale=.2]{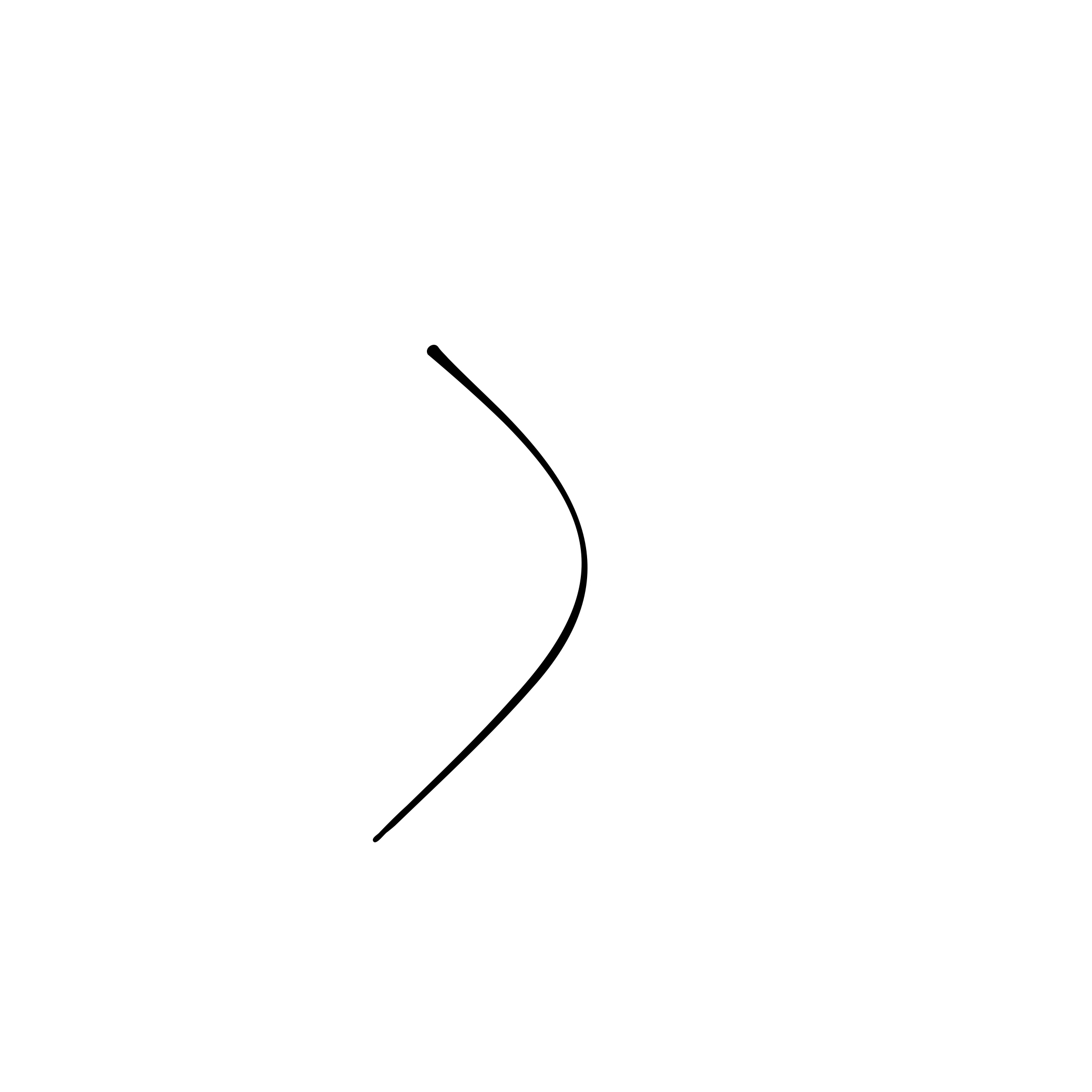}};
\node[image,right=of A] (B) {\includegraphics[scale=.2]{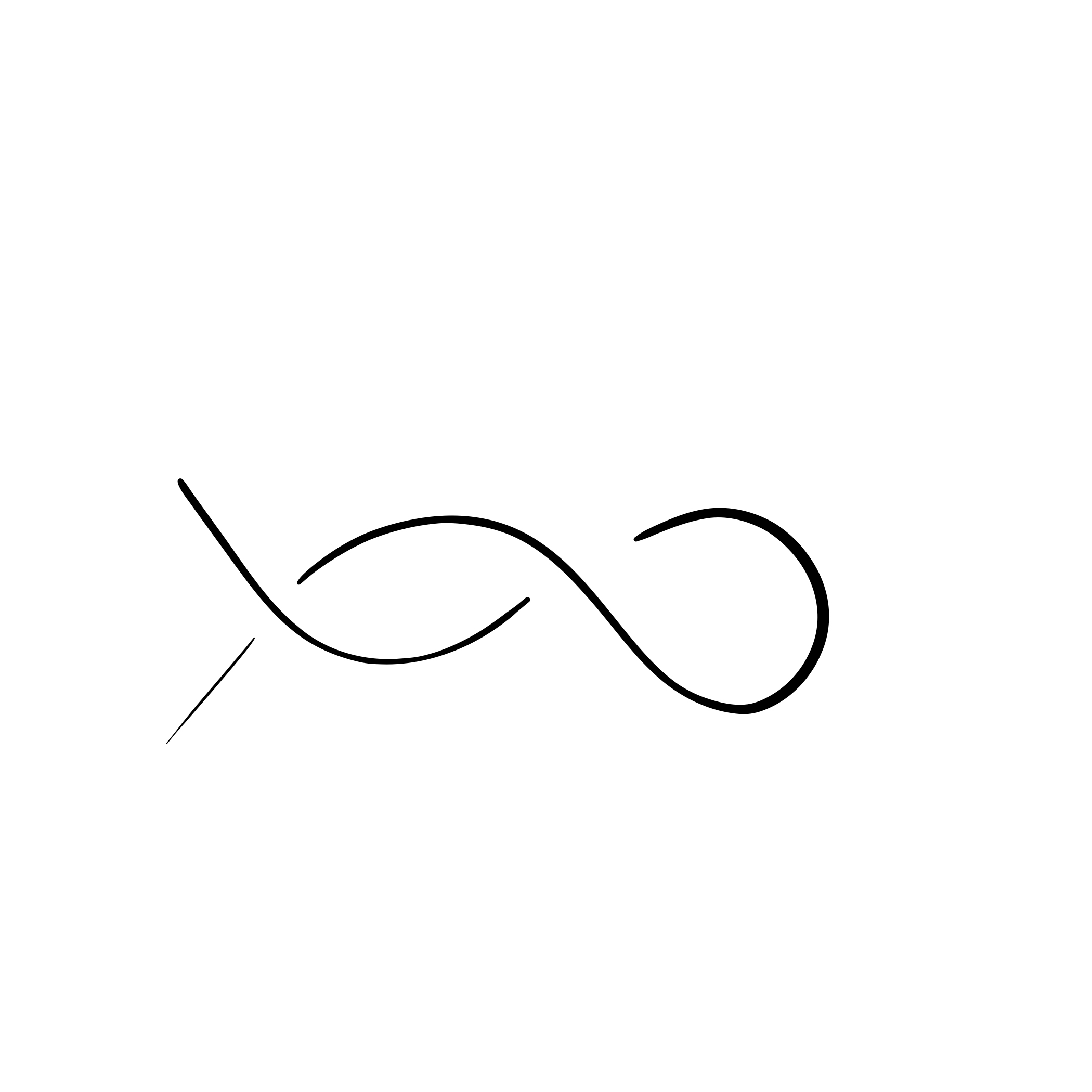}};
\draw[line width=0.30mm, ->](1.5,-.3) --(3,-.3);
\end{tikzpicture}
    \caption{Stabilization of the transverse knot $K$.}
    \label{TransStab}
\end{figure}

A Legendrian knot $L$ can be perturbed to a canonical (up to transverse isotopy) transverse link $K$ called the \textit{transverse pushoff}. Legendrian isotopic links give rise to transversely isotopic push-offs. Also, any transverse link $K$ will be transversely isotopic to the push-off of a Legendrian link $L$. We say that $L$ is a \textit{Legendrian approximation} of $K$. The transverse push-off of a Legendrian link is transversely isotopic to the push-off of its negative stabilization. On the other hand, any two Legendrian approximations of a transverse link are Legendrian isotopic after a series of negative Legendrian stabilizations. For a complete description of the relationship between transverse and Legendrian knots, see \cite{Ge}.

\section{Classical Invariants of Legendrian and Transverse Knots}\label{Classical}

In this section, we first introduce the classical invariants for Legendrian knots and then the classical invariant of transverse knots. For a more complete description of the classical invariants of Legendrian and transverse knots, the reader is referred to \cite{Et1,Ge}.

Let $K$ be a homologically trivial Legendrian knot in a contact 3-manifold $(M, \xi)$. The \textit{Thurston-Bennequin invariant} of $K$, denoted by $\mathrm{tb}(K)$, is the twisting of the contact framing relative to the surface framing of $K$, with right-handed twists being counted positively. In other words, if we choose a vector field along $K$ transverse to $\xi$ and define a parallel knot $K'$ by pushing $K$ along this vector field, then $\mathrm{tb}(K)$ equals the linking number of $\mathrm{lk}(K,K')$ of $K$ with $K'$. In Section~\ref{Proof}, we will use the Thurston-Bennequin invariant in the proof of the main result of this article. We can compute the Thurton-Bennequin invariant given a front projection, $\Pi(K)$, of the Legendrian knot $K$. Thus, the Thurston-Bennequin invariant is the following: 
\[ tb(K) = \textup{writhe}(\Pi(K)) - \dfrac{1}{2}(\textup{number of cusps in } \Pi(K)).\]
Legendrian knots also come equipped with the \emph{rotation number}. To see an interpretation of this invariant, the reader is referred to \cite{Et1}. Similar to the Thurston-Bennequin invariant, we can compute the rotation number with the front projection of the Legendrian knot $K$ using the following formula: 
\[ r(K) = \frac{1}{2}(D - U),\]
where $U$ is the number of up cusps in the front projection of $K$, and $D$ is the number of cusps in the front projection of $K$.

\section{Transverse Knots and Braids}\label{Braids}
In this section, we will explore the relationship between transverse knots and braids with the aim of thinking of a transverse knot as the closure of a braid. For a more complete description of the relationship between transverse knots and braids, we refer the reader to \cite{BVV, Et1, HaKaPl, Pl}. 

To explore the connection between transverse knots and closed braids, we will consider $(S^3, \xi_{std})$. Bennequin in \cite{Be}, proved the following:

\begin{theorem}
Any transverse knot in $(S^3, \xi_{std})$ is transversely isotopic to a closed braid. 
\end{theorem}

Recall that the usual Markov Theorem states a sequence of conjugations and braid stabilizations relates two braids representing the same knot or link. Orevkov and Shevchishin \cite{OrSh} and independently 
Wrinkle \cite{Wr}, proved the following transverse Markov theorem.

\begin{theorem}
Two braids represent the same transverse knot if and only if they are related by positive stabilization and conjugation in the braid group. 
\end{theorem}

\begin{remark}
Let $K$ be a transverse knot. The stabilization of $K$, denoted by $K_{stab}$, represented by a braid, is equivalent to the negative braid stabilization, which adds an extra strand and a negative kink to the braid.
 Furthermore, the positive braid stabilization does not change the transverse type of the knot or link. 
\end{remark}

\section{$n$-fold Cyclic Branched Covers} \label{Branched}
This section is to introduce the basic definitions and constructions of branched covers of contact 3-manifolds. For a more detailed description of the theory of branched covers of 3-manifolds, the reader is referred to \cite{Ro}, and for a complete description of branched covers and contact 3-manifolds, the reader is referred to \cite{Ca, HaKaPl}.
\begin{definition}
Let $M$ and $N$ be 3-manifolds. A map $p: M \rightarrow N$ is called a branched covering if there exists 
a dimension 1 complex $K$ such that $p^{-1}(K)$ is a dimension 1 complex and $p|_{M-p^{-1}(K)}$ is a 
covering. 
\end{definition}
One can think of a branched covering as a map between manifolds such that, away from the \textit{branch locus}, a codimension 2 set, $p$ is a covering.

Now, we will describe branched covers of contact manifolds. Let $K$ be a transverse knot in $(Y, \xi)$ and branched covering $p : \tilde{Y}:=\Sigma_n(K) \rightarrow Y$ with branch locus $K$. Now the question becomes, how do we define the lift of $\xi$? First, the manifold $\Sigma_n (K)$ is constructed as usual. Let $\widetilde{K} = p^{-1}(K)$. Then $p : (\Sigma_n (K)-\widetilde{K}) \rightarrow Y - K$ is a true cover, and thus $\tilde{\xi} = p^*(\xi)$ on $(\Sigma_n(K) - \widetilde{K})$. Gonzalo in \cite{Go} extended the contact structure by carefully analyzing the branched covering map near the branch locus $K$ over $Y$. Therefore, we obtain a contact structure $\xi_n$ on $\Sigma_n(K)$, and we will denote the contact manifold obtained from this construction by $(\Sigma_n(K), \xi_n)$.

\section{Stabilization result}\label{Proof}

In this section, we reprove a result initially demonstrated by Harvey, Kawamuro, and Plamenvskaya in \cite{HaKaPl}. While the original argument relied on contact surgery, we will explicitly construct the overtwisted disk. Additionally, by constructing the overtwisted disk, we show that it is contained in the complement of the branch locus of the $n$-fold cyclic branched cover.

\begin{theorem}\label{stab}
If $K_{stab}$ is a negative braid stabilization of $K$, then $(\Sigma_n(K_{stab})-\widetilde{K_{stab}}, \xi_n)$ is overtwisted.   
\end{theorem}

\begin{proof}
Given a transverse knot $K$ in $(S^3, \xi_{std})$, we can uniquely approximate $K$ by a Legendrian knot $L$ up to negative Legendrian stabilization. Given a Legendrian approximation of $K$, we can perform a negative stabilization locally and still approximate $K$; see Figure~\ref{NegStab}.
\begin{figure}[h]
    \begin{tikzpicture}[
node distance = 3mm,
image/.style = {scale=0.3,},
arr/.style = {-Triangle, semithick}
]
\node[image] (A) {\includegraphics[scale=.2]{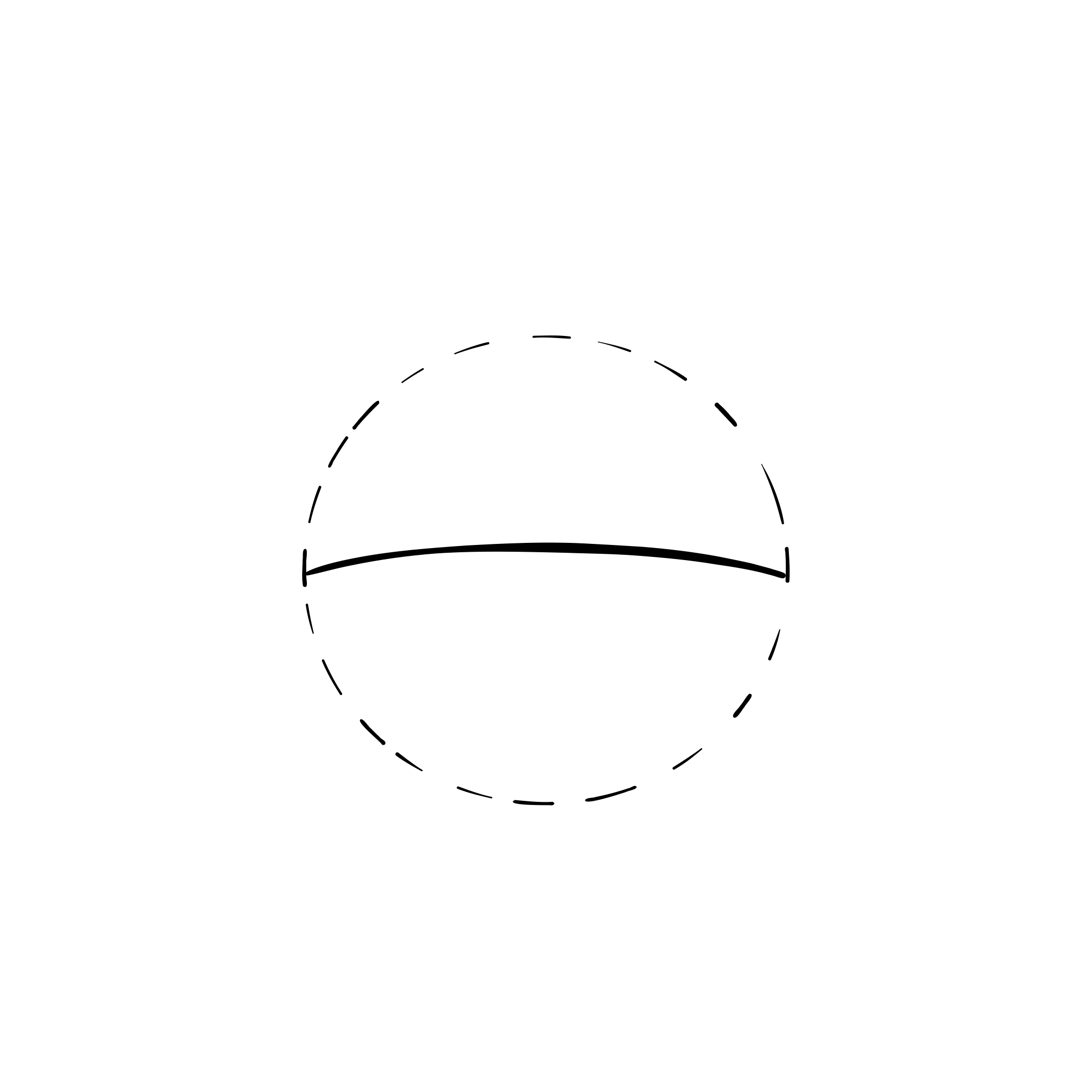}};
\node[image,right= 5mm of A] (B) {\includegraphics[scale=.2]{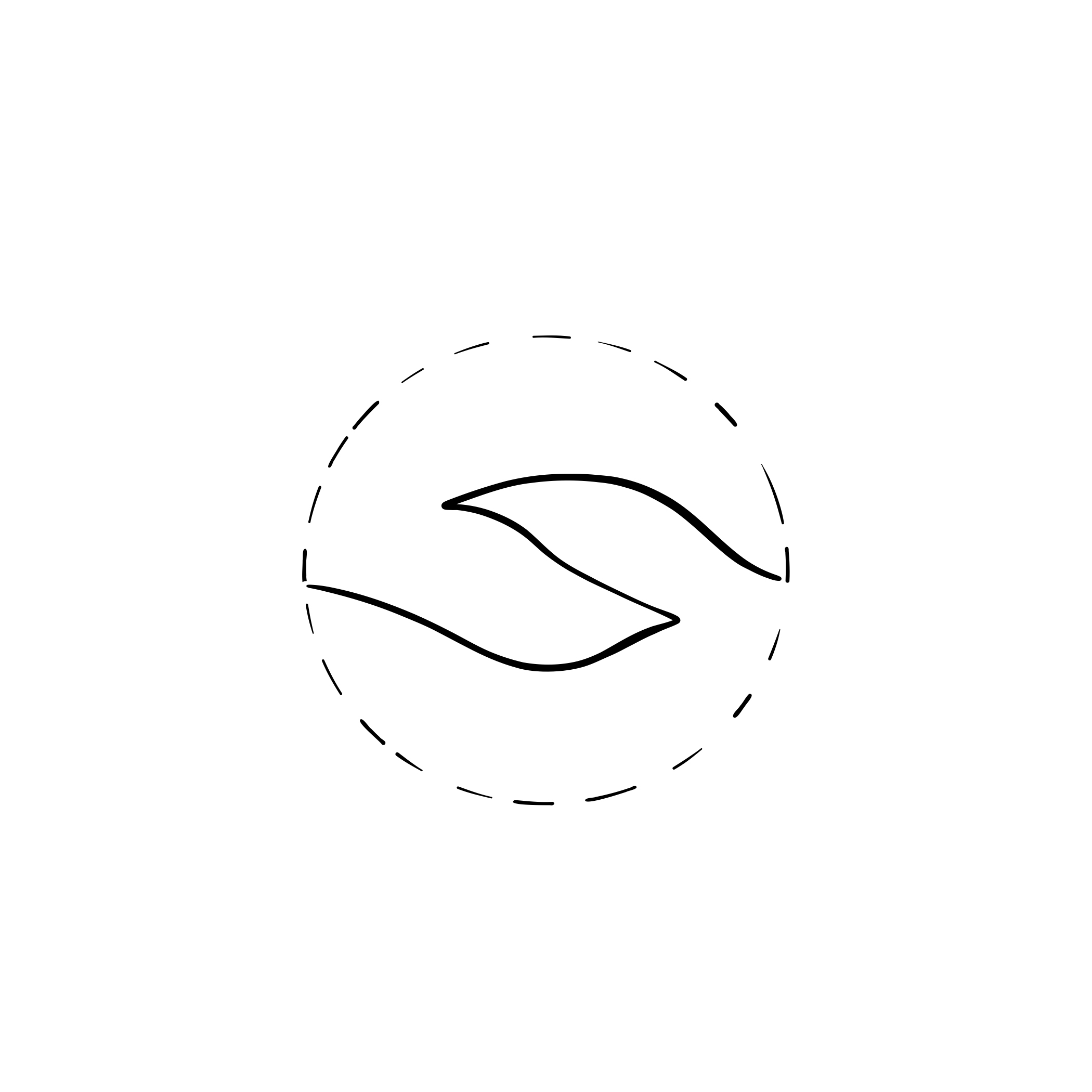}};
\draw[line width=0.30mm, ->](A) --(B);
\end{tikzpicture}
\caption{The figure on th left is the Legendrian approximation $L$ of $K$, the figure on the right is another Legendrian approximation $L'$ after performing a negative Legendrian stabilization.}
\label{NegStab}
\end{figure}
Since transverse knots do not detect negative Legendrian stabilization, both of the above Legendrian knots are Legendrian approximations of the transverse knot $K$.  Now, suppose that $K_{stab}$ is the stabilization of the transverse knot $K$. From the negatively stabilized Legendrian knot $L'$, we can also obtain a Legendrian approximation of the stabilization of $K_{stab}$. To get a Legendrian approximation of $K_{stab}$, we must perform a positive Legendrian stabilization to $L'$.


\begin{figure}
\includegraphics[scale=.05]{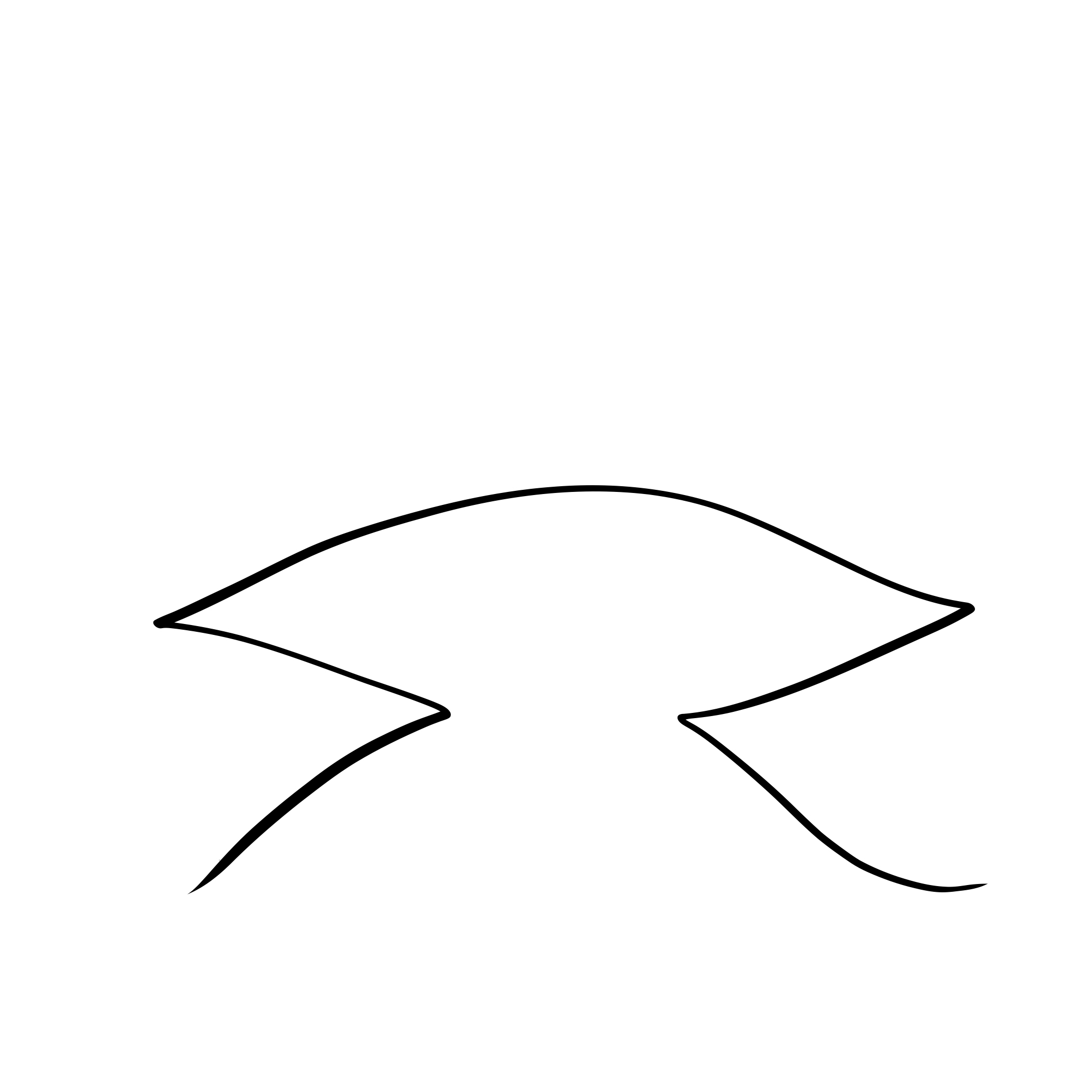}
\caption{A local picture of $L'_{+}$.}
\label{PositiveStab}
\end{figure}
We will denote the positive Legendrian stabilization of $L'$ by $L'_+$. In Figure~\ref{PositiveStab}, we see both the positive and negative Legendrian stabilizations. Consider the curve $\gamma$ in Figure~\ref{gamma}.

\begin{figure}[h]
    \begin{tikzpicture}[
node distance = 3mm,
image/.style = {scale=0.3,},
arr/.style = {-Triangle, semithick}
]
\node[image] (A) {\includegraphics[scale=.18]{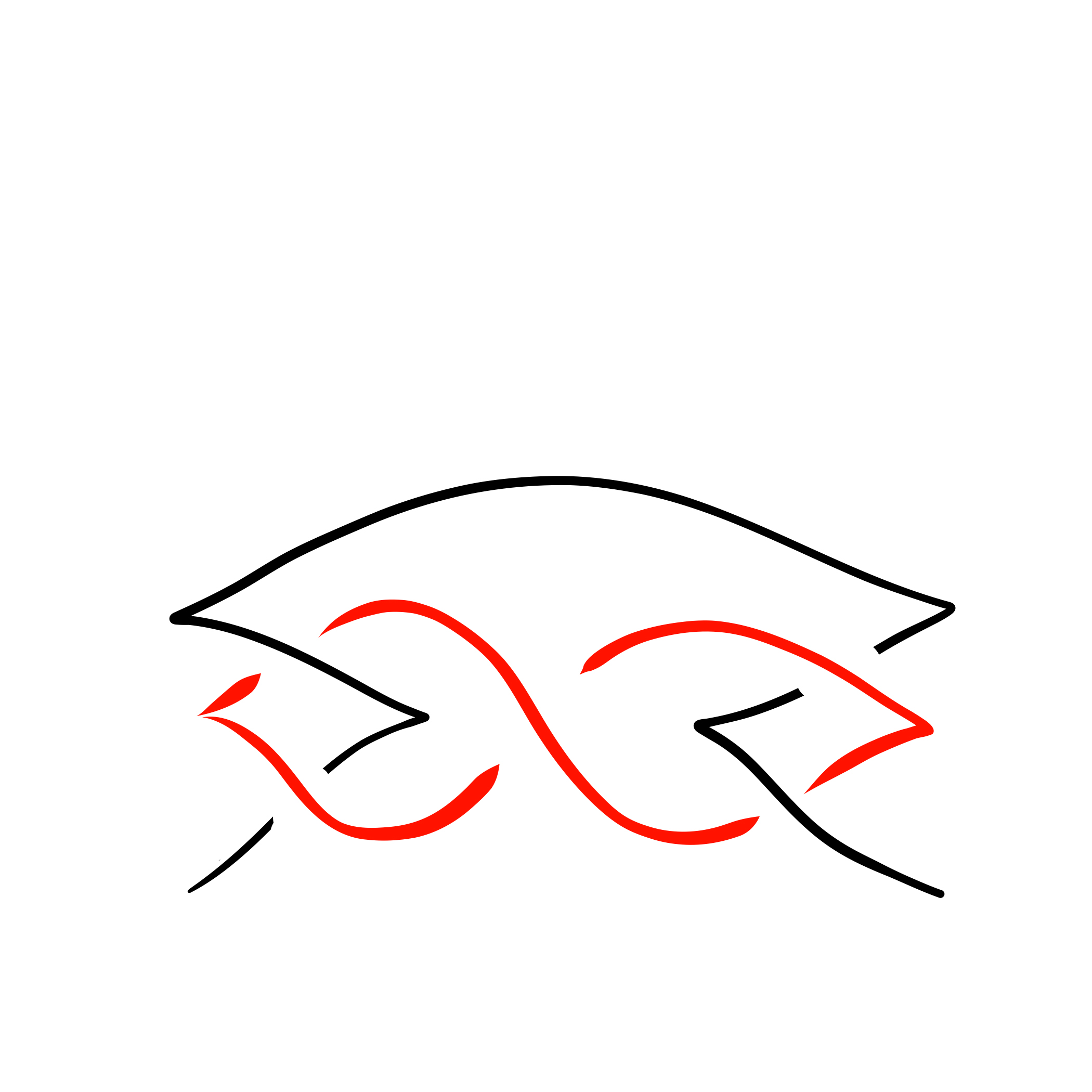}};
\node [draw=none, inner sep = 0] at (2,-.8) {\large{$\gamma$}};

\end{tikzpicture}
\caption{The curve $\gamma$.}
\label{gamma}
\end{figure}
Since we can always perform the above local operations on a Legendrian approximation of $K$. We want to think of these local pictures separately. We can think of the appropriate Legendrian unknot depicted in Figure~\ref{unknot} and then connect sum the Legendrian unknot with any Legendrian knot to get Figure~\ref{gamma}.
\begin{figure}[h]
\begin{tikzpicture}[
node distance = 3mm,
image/.style = {scale=0.3,},
arr/.style = {-Triangle, semithick}
]
\node[image] (A) {\includegraphics[scale=.2]{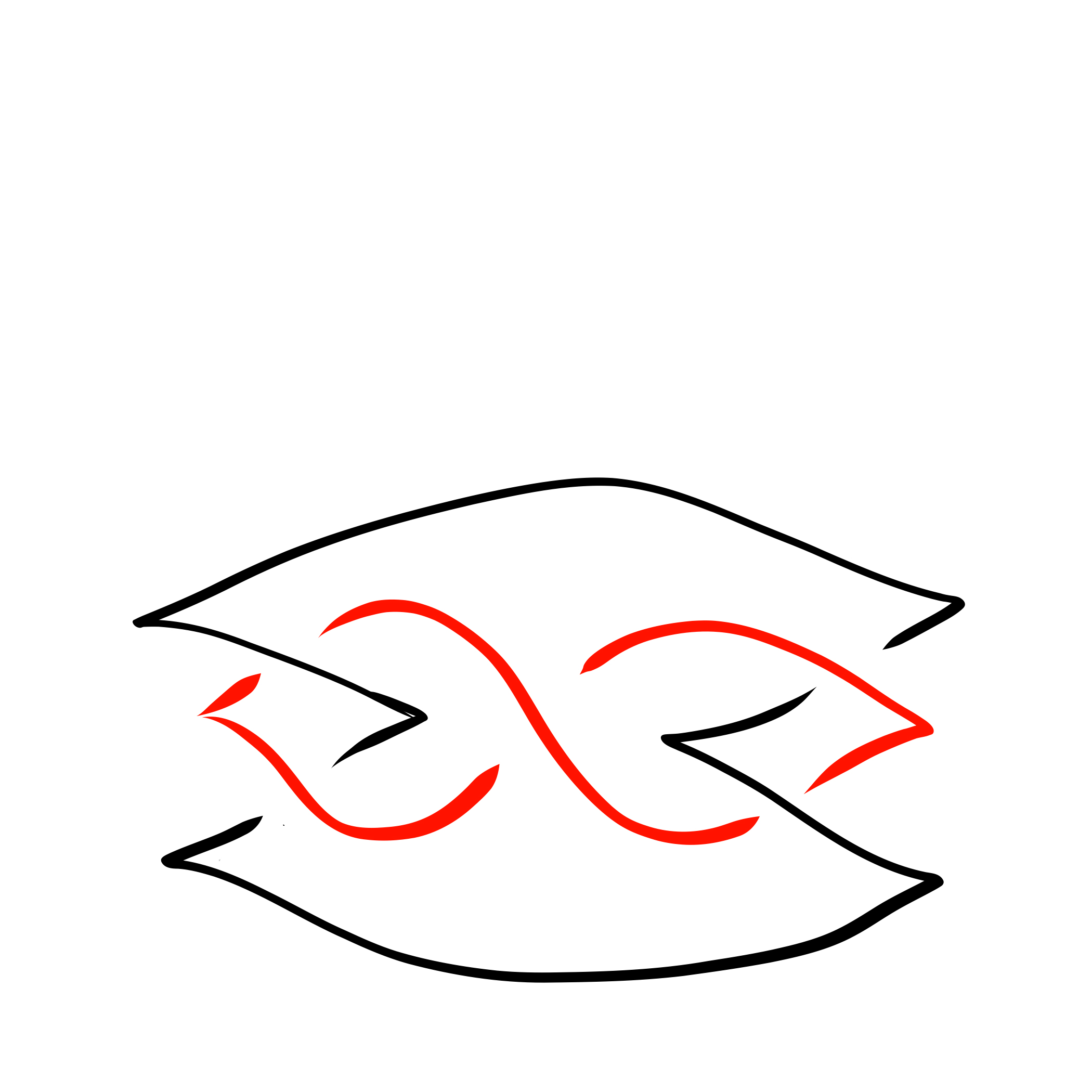}};
    \node [draw=none, inner sep = 0] at (2,-.8) {\large{$\gamma$}};
\end{tikzpicture}
\caption{Legendrian unknot and the curve $\gamma$.}
\label{unknot}
\end{figure}
Note that $\gamma$ has $\mathrm{tb}(\gamma) = -2$ in $S^3$. Topologically, we have Figure~\ref{smooth}.  
\begin{figure}[h]
\begin{tikzpicture}[
node distance = 3mm,
image/.style = {scale=0.3,},
arr/.style = {-Triangle, semithick}
]
\node[image] (A) {\includegraphics[scale=.18]{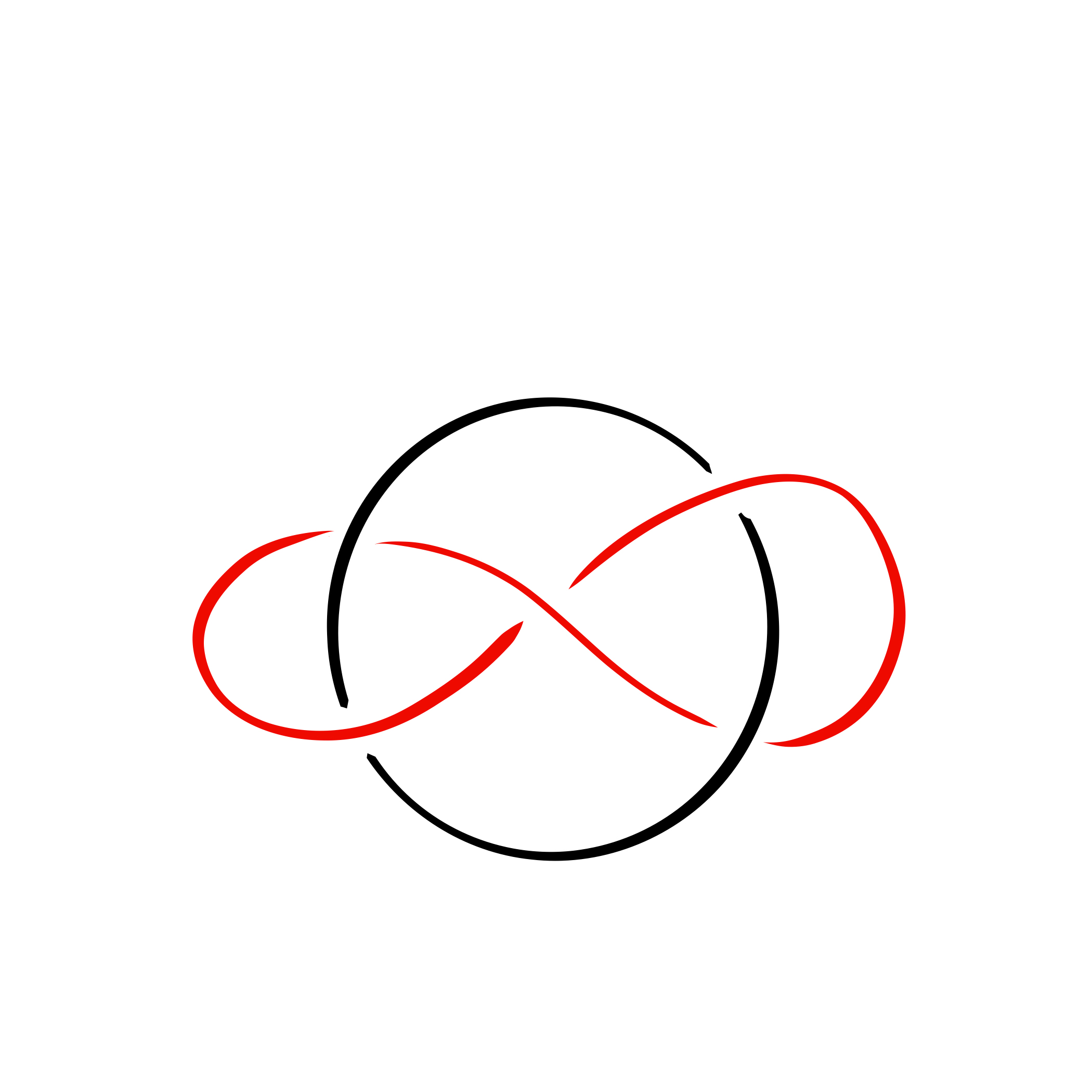}};
    \node [draw=none, inner sep = 0] at (2,-.5) {\large{$\gamma$}};
\end{tikzpicture}
\caption{The Legendrian link in Figure~\ref{unknot} as a smooth link.}
\label{smooth}
\end{figure}
Furthermore, we can isotope $\gamma$ to obtain the negative Whitehead link; see Figure \ref{whitehead}.
\begin{figure}[h]
\begin{tikzpicture}[
node distance = 3mm,
image/.style = {scale=0.3,},
arr/.style = {-Triangle, semithick}
]
\node[image] (A) {\includegraphics[scale=.2]{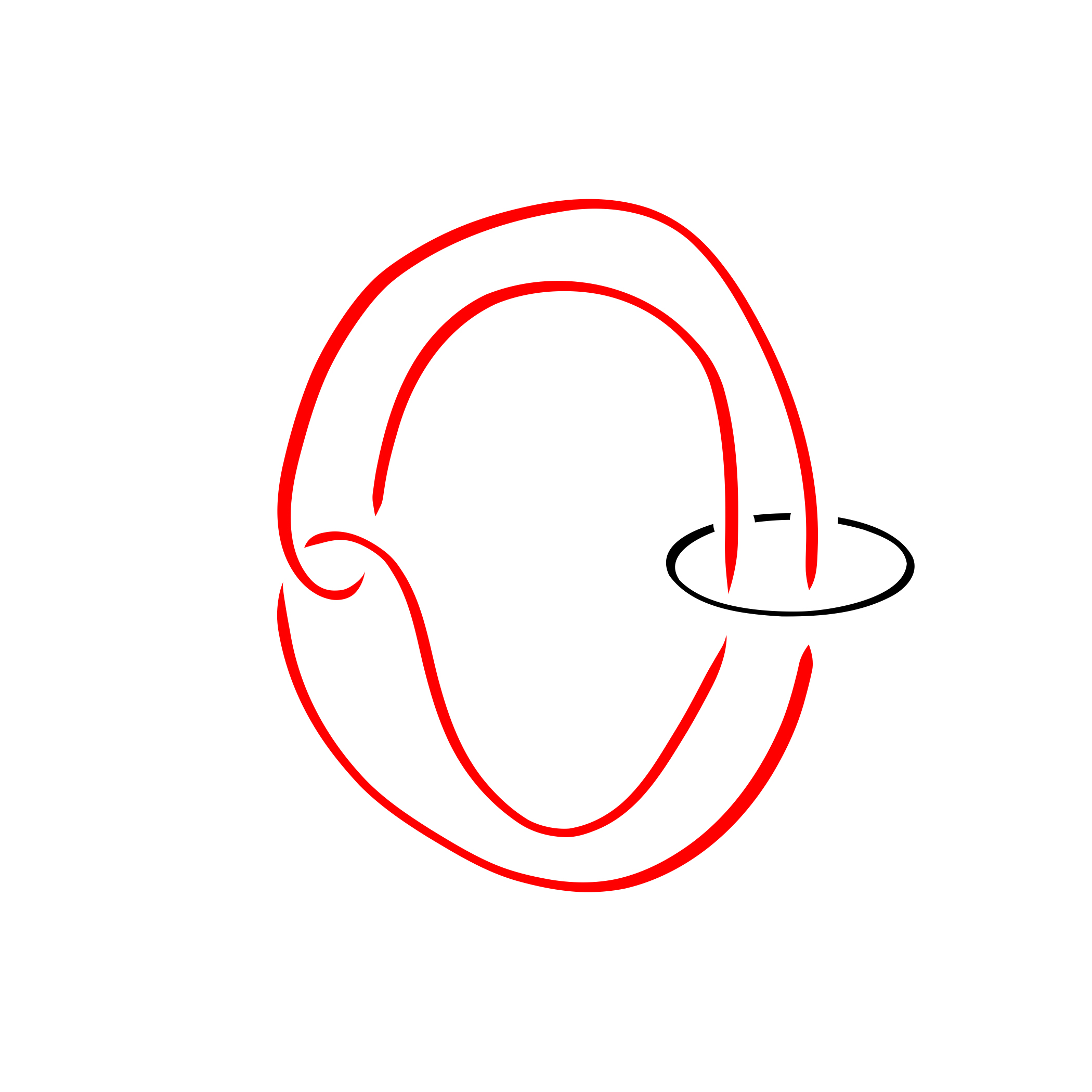}};
    \node [draw=none, inner sep = 0] at (-1.5,0) {\large{$\gamma$}};
\end{tikzpicture}
\caption{Diagram of the (negative) Whitehead link.}
\label{whitehead}
\end{figure}
We see that $\gamma$ bounds a disk $D$ depicted in the following figure by the gray region. We have $lk(\gamma,\gamma')=-2$ where $\gamma'$ is the push off of $\gamma$ obtained by sliding $\gamma$ into the interior of the disk $D$. Hence, the twisting of $D$ with respect to the contact planes is zero. Therefore, if $D$ were to be embedded, it would be an overtwisted disk. This is impossible, so this means that $\gamma$ bounds an immersed disk $D$ disjoint from the link; see Figure~\ref{immersed}. 
\begin{figure}[h]
\begin{tikzpicture}[
node distance = 3mm,
image/.style = {scale=0.3,},
arr/.style = {-Triangle, semithick}
]
\node[image] (A) {\includegraphics[scale=.2]{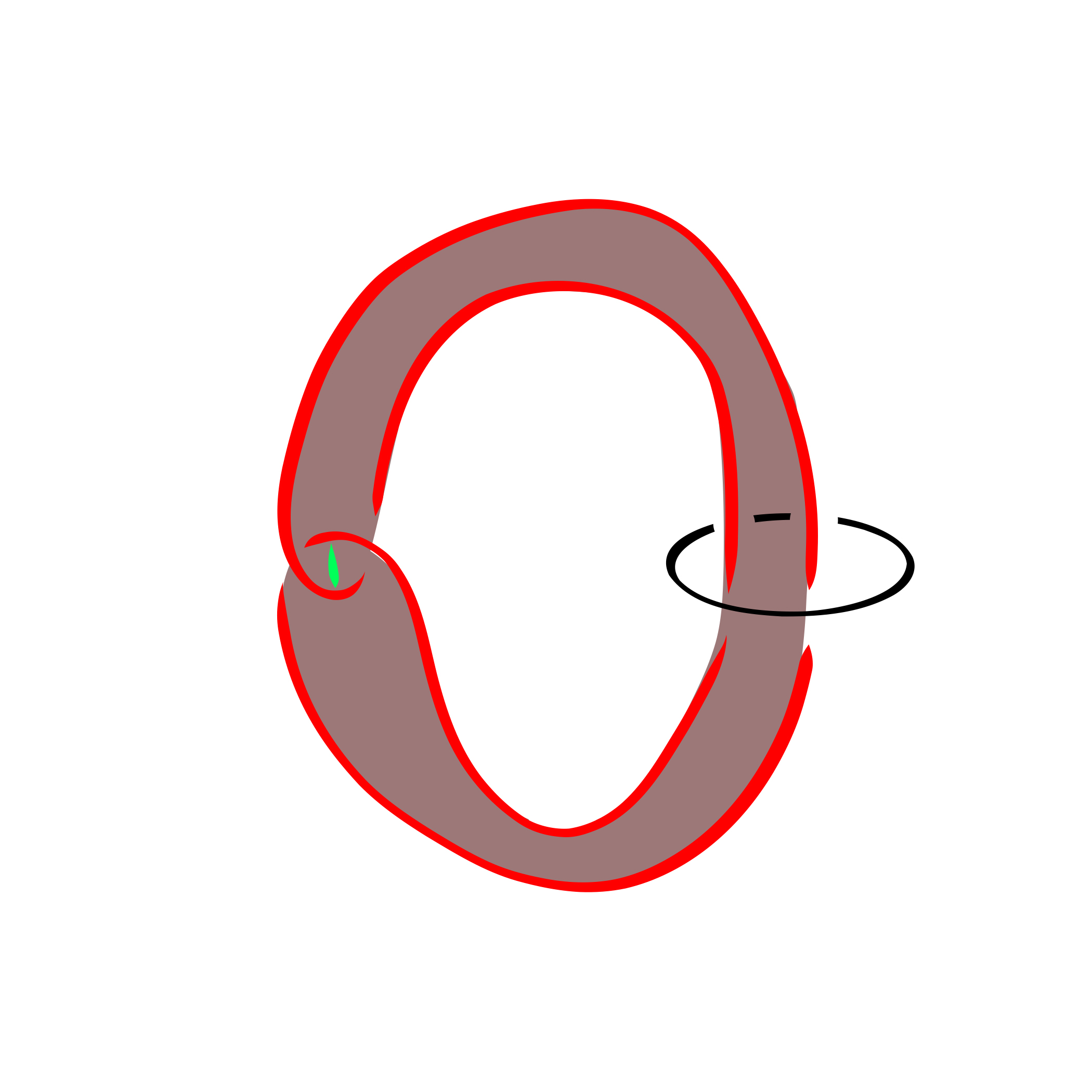}};
    \node [draw=none, inner sep = 0] at (-1.5,0) {\large{$\gamma$}};
        \node [draw=none, inner sep = 0] at (-.7,-.5) {{\color{white}$D$}};
\end{tikzpicture}
\caption{Immersed disk $D$ bounded by $\gamma$.}
\label{immersed}
\end{figure}
But if we consider the double branch cover $\pi : \Sigma_2(K_{stab}) \rightarrow S^3$ the disk $D$ lifts to two embedded, overtwisted disks in $\Sigma_2(K_{stab})$, see Figure~\ref{lift}. 
\begin{figure}[h]
\includegraphics[scale=.06]{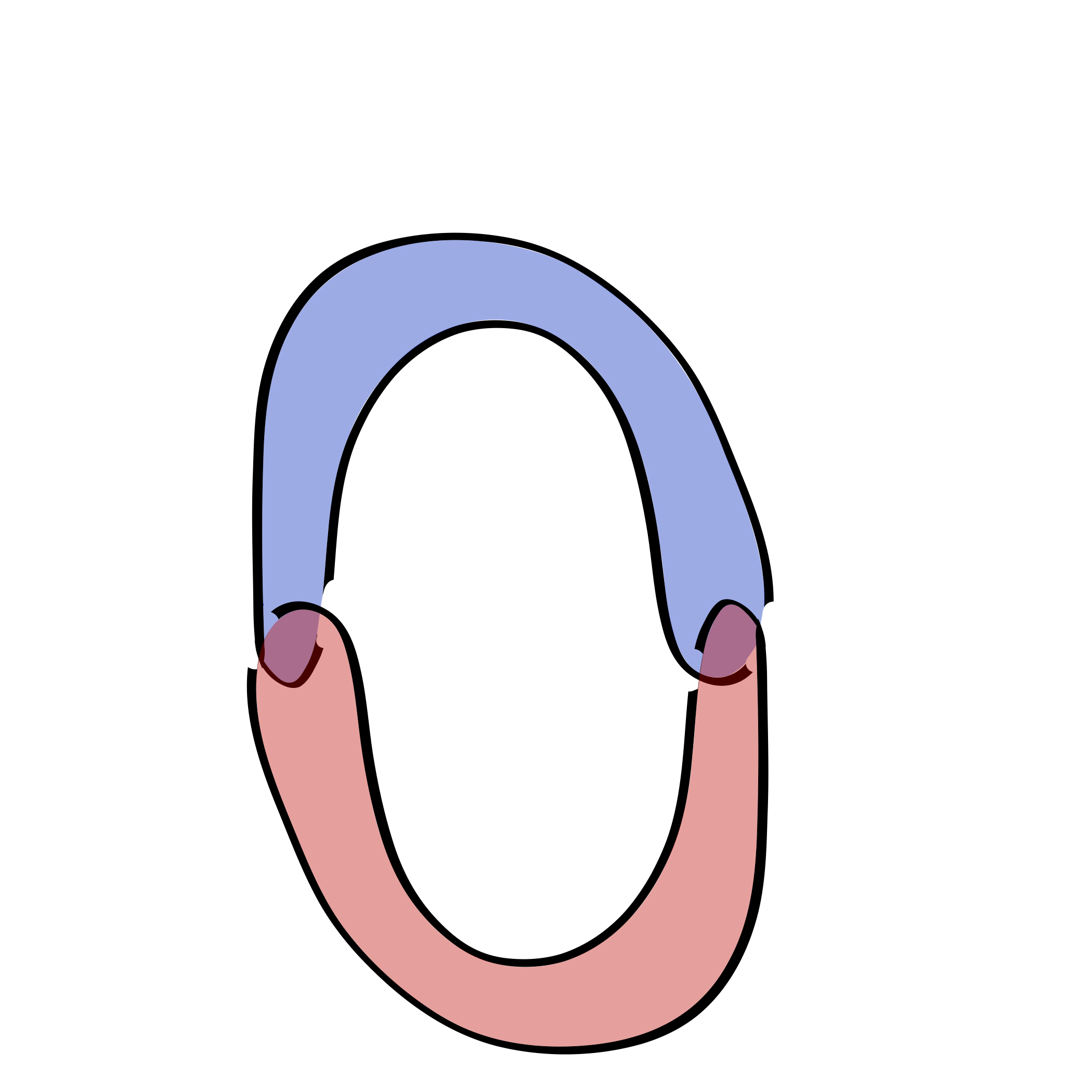}
\caption{Lift of $D$ in $\Sigma_2(K_{stab})$.}
\label{lift}
\end{figure}
If we consider the $n$-fold cyclic branch cover branched along $K_{stab}$, we will obtain $n$ embedded disks with all of them being overtwisted. Therefore, any $n$-fold cyclic branch cover with branched locus $K_{stab}$ is overtwisted.  
\end{proof}

\begin{corollary}
If $K_{stab}$ is a negative braid stabilization of $K$, then $(\Sigma_n(K_{stab}), \xi_n)$ is overtwisted.   
\end{corollary}

\begin{remark}
Theorem~\ref{stab} states more than that the $n$-fold cyclic branched cover branched along $K_{stab}$ is overtwisted. Using this argument, we see that the overtwisted disks stay away from the branch locus and are wholly contained in the complement. 
\end{remark}

\section{Future Research}
The authors of \cite{HaKaPl} were interested in distinguishing transverse knots by focusing on the contact structure obtained from the branched cover branched along the transverse knots. They were able to find an example of non-isotopic knots with non-contactomorphic branched covers. Additionally, they found several examples of smoothly isotopic transverse knots with the same self-linking number and contactomorphic branched covers. However, they could not distinguish transverse knots by their $n$-fold cyclic covers. Another unexplored possibility is to consider any potential information that can be obtained from the lift of the transverse knot. In a future project, we will consider the lift $\tilde{K}$ of the transverse knot $K$ in $(\Sigma_n(K), \xi_n)$.

\section*{Acknowledgement} 

In memory of Dr. Richard A. Litherland. Dr. Litherland made significant use of branched covers in his work, most notably in his collaboration with Dr. Cameron Gordon in \cite{GL}, where the Gordon-Litherland pairing was introduced. He also explored how different graph colorings affect the homology of covers in \cite{L}. Additionally, I am forever grateful for his Introduction to Topology course at Louisiana State University, where I first discovered my passion for topology. I would also like to thank my graduate advisor, Dr. David Shea Vela-Vick, for providing valuable advice and guidance.


\bibliography{References}
\bibliographystyle{plain}

\end{document}